\documentclass[12pt]{amsart}
\usepackage{amssymb}
\usepackage{amscd}
\usepackage[ruled,lined]{algorithm2e}
\usepackage{array}
\usepackage{xy}
\usepackage{rotating}

\xyoption{all}

\newtheorem{thm}{Theorem}

\newtheorem{lem}[thm]{Lemma}
\newtheorem{cor}[thm]{Corollary}
\newtheorem{prop}[thm]{Proposition}
\newtheorem{defn}[thm]{Definition}

\newcommand{\GL}{\operatorname{GL}}

\newcommand{\Z}{{\mathbb Z}}
\newcommand{\z}{{\mathbb Z}}

\newcommand{\q}{{\mathbb Q}}
\newcommand{\HH}{{\mathfrak H}}
\newcommand{\C}{{\mathbb C}}

\newcommand{\f}{{\mathbb F}}

\def\q{{\mathbb Q}}
\def\z{{\mathbb Z}}

\def\C{{\mathfrak C}}

\def\f{{\mathbb F}}

\def\Pr{{\mathbb P}}

\usepackage{tikz}
\providecommand{\myfloor}[1]{\left \lfloor #1 \right \rfloor }

\begin{document}

\title[An explicit correspondence of modular curves]
{An explicit correspondence of modular curves}

\author{Imin Chen and Parinaz Salari Sharif}

\date{December 2017}

\address{Imin Chen \\
Department of Mathematics \\
Simon Fraser University \\
Burnaby \\
British Columbia \\
CANADA.}

\email{ichen@sfu.ca}

\address{Parinaz Salari Sharif \\
Department of Mathematics \\
Simon Fraser University \\
Burnaby \\
British Columbia \\
CANADA.}

\email{psalaris@sfu.ca}

\thanks{Research supported by an NSERC Discovery Grant and a SFU VPR Bridging Grant.}

\begin{abstract}
In this paper, we recall an alternative proof of Merel's conjecture which
asserts that a certain explicit correspondence gives the isogeny relation
between the Jacobians associated to the normalizer of split and non-split
Cartan subgroups. This alternative proof does not require extensive
representation theory and can be formulated in terms of certain finite
geometries modulo $\ell$.

Secondly, we generalize these arguments to exhibit an explicit correspondence which gives
the isogeny relation between the Jacobians associated to split and non-split Cartan subgroups.
An interesting feature is that the required explicit correspondence is considerably more complicated but
can expressed as a certain linear combination of double coset operators whose coefficients
we are able to make explicit.
\end{abstract}

\maketitle

\section{Introduction}

Modular curves, which are coarse moduli spaces for elliptic curves with
prescribed level structure, appear in the study of Galois torsion structures
on elliptic curves.

A well-known example is Mazur's Theorem \cite{mazur} which states that there
are no rational $\ell$-isogenies between rational elliptic curves if $\ell >
163$. This result is proven by showing the modular curve $X_0(\ell)$ has no
non-cuspidal rational point if $\ell > 163$. Mazur's method is based on
descent on the Jacobian of $X_0(\ell)$, but because of the rich arithmetic
structure of these curves, the method is more powerful and efficient.

Let $\ell$ be a prime, and $\z/\ell\z = \f_\ell$ be a finite field of cardinality $\ell$.

For a subgroup $H$ of $\GL_2(\f_\ell)$ which contains $-1$, it is possible
to associate a modular curve $X_H := X/H$. In the case when $H$ is a non-split
Cartan subgroup $C'$ or its normalizer $N'$, it is relevant from the point of
view of Mazur's method to understand the Jacobian of $X_H$. In
\cite{chen-98}, it was proven using the trace formula that $X_{N'}$ and
$X_{C'}$ are related by an isogeny over $\q$ to certain quotients of the
Jacobian of the modular curves $X_0(\ell^2)$. Subsequently, 
a proof based on the representation theory of $\GL_2(\f_\ell)$ was given in \cite{edixhoven}.

In \cite{merel}, it was conjectured that the above isogeny relation between
the Jacobian of $X_{N'}$ and the Jacobian of $X_0(\ell^2)$ was given by a
certain explicit correspondence. This was proven in \cite{chen} using the
representation theory of $\GL_2(\f_\ell)$ and identities in finite double coset algebras.

In this paper, we recall an alternative proof of Merel's conjecture, which
does not require extensive representation theory, based on arguments given by
B.\ Birch and D.\ Zagier \cite{zagier}. The proof can be formulated in terms
of certain finite geometries over $\f_\ell$ and is largely elementary in
its statement and proof, though some algebraic number theory is used.

Secondly, we generalize these arguments to exhibit an explicit correspondence
which gives the isogeny relation between the Jacobians associated to split and
non-split Cartan subgroups. An interesting
feature is that the required explicit correspondence is considerably more
complicated but can be expressed as a certain linear combination of double coset
operators whose coefficients we are able to make explicit.

The precise statements of the theorems we prove are as follows.
\begin{itemize}
\item Let $\ell$ be an odd prime and $\epsilon$ a non-square in
$\f_\ell^\times$.

\item Let $G = \GL_2(\f_\ell)$.

\item Let $\Pr^1(\f_\ell) \times \Pr^1(\f_\ell) - \Delta$ denote the set
of ordered pairs $(a,b)$ of distinct points in $\Pr^1(\f_\ell)$.

\item Let $(\Pr^1(\f_\ell) \times \Pr^1(\f_\ell) -\Delta)/\sim$, where $(a,b) \sim (b,a)$,  denote the set of unordered pairs
$\left\{a,b\right\}$ of distinct points in $\Pr^1(\f_\ell)$.

\item Let $\C_\ell = \left\{ x + y \sqrt{\epsilon} : x \in \f_\ell, y \in
\f_\ell^\times \right\}$.

\item Let $\HH_\ell = \C_\ell/\sim$, where $x + y \sqrt{\epsilon} \sim x -
y \sqrt{\epsilon}$.

\item When $S$ is a set, we denote by $\q[S]$ the free $\q$-vector space generated by the set $S$.

\item For convenience, we write column vectors in the form $(x,y)^t$ for instance.
\end{itemize}

Given an unordered pair $\left\{ a, b \right\}$ in
$(\Pr^1(\f_\ell) \times \Pr^1(\f_\ell) -\Delta)/\sim$, we define in~\eqref{geodesic} a `geodesic'
$\gamma_{\left\{a, b \right\}}$ in $\HH_\ell$ between $a$ and $b$. 
\begin{thm}
\label{main-N} The map
\begin{align*}
\label{main-Neq}\psi^+ : \q[(\Pr^1(\f_\ell) \times \Pr^1(\f_\ell) - \Delta)/\sim] & \rightarrow \q[\HH_\ell] \\
\notag \left\{ a,b \right\} & \mapsto \sum_{x \in \gamma_{\left\{a,b\right\}}} x
\end{align*}

is a surjective $\q[G]$-module homomorphism.
\end{thm}

Given an ordered pair $(a,b)$ in $\Pr^1(\f_\ell) \times
\Pr^1(\f_\ell) - \Delta$ and a parameter $s \in \f_\ell^\times$, we define in~\eqref{path} a `path' $\gamma^s _{(a,b)}$ in
$\C_\ell$ from $a$ to $b$. 
\begin{thm}
\label{main-C} The map
\begin{align*}
\psi : \q[\Pr^1(\f_\ell) \times \Pr^1(\f_\ell) - \Delta] & \rightarrow \q[\C_\ell] \\
(a,b) & \mapsto \sum_{s=1}^{\ell-1}(\alpha_s + \beta_s) \sum_{x \in \gamma^s_{(a,b)}} x
\end{align*}
is a surjective $\q[G]$-module homomorphism, where $0 \le \alpha_s, \beta_s \le \ell-1$ are integers satisfying $\alpha_s \equiv 1\pod \ell$ and $\beta_s \equiv s^{-1} \pod \ell$ for $s \in \left\{ 1, \ldots, \ell-1 \right\}$.
\end{thm}

We explain in section~\ref{lastchapter} how Theorems~\ref{main-N} and \ref{main-C}
imply relations between the Jacobians of $X_{N'}$ and $X_{C'}$ and quotients of the Jacobians of the more standard modular curve $X_0(\ell^2)$.

\section*{Acknowledgments}

We would like to thank B.\ Birch and D.\ Zagier for explaining to us their alternative proof of Merel's conjecture.

\section{Double coset operators}

\begin{lem}
\label{doublecoset}
Let $G$ be a group, $H$ and $K$ be subgroups of $G$, then
\begin{align*}
& HgK = \bigcup_{\alpha \in H/H \cap g K g^{-1}} \alpha g K,
\end{align*}
where the union is disjoint. We call $[H : H \cap g K g^{-1}]$ the degree of $HgK$. This is independent of the choice of $g$ in the sense that $deg(HgK)= deg(Hg'K)$ if $HgK=Hg'K$.
\end{lem}

\begin{defn} 
Let $G$ be a finite group with subgroups $H$ and $K$. Given a double coset $H g K$ and a decomposition into disjoint cosets
\begin{equation*}
   H g K = \bigcup_{\alpha \in \Omega} \alpha g K,
\end{equation*}
we obtain a $\z[G]$-module homomorphism $\sigma = \sigma(H g K)$ given by
\begin{align}
\label{doublecoset-defn}
\sigma  : \z[G/H] & \rightarrow \z[G/K] \\
\notag   x H & \mapsto \sum_{\alpha \in \Omega} x \alpha g K.   
\end{align}
The $\z[G]$-module homomorphism $\sigma$ is called a double coset operator.
\end{defn}
 
Let $C$ (resp.\ $C'$) be the split (resp.\ non-split) Cartan subgroup of $G$ given respectively by
\begin{align*}
 & C= \left\{ \begin{pmatrix} \eta & 0 \\ 0 & \beta \end{pmatrix} : \eta , \beta \in \f_\ell ^{\times} \right\}, \\
 & C' = \left\{ \begin{pmatrix} x & \epsilon y \\ y & x \end{pmatrix} : (x, y) \neq (0, 0), x, y \in \f_\ell \right\}.
\end{align*}

Let $N$ (resp.\ $N'$) be the normalizer in $G$ of $C$ (resp. $C'$) which is given respectively by
\begin{align*}
 & N= \left\{ \begin{pmatrix} \eta & 0 \\ 0 & \beta \end{pmatrix}, \begin{pmatrix} 0 & \eta \\ \beta & 0 \end{pmatrix}  : \eta , \beta \in \f_\ell ^{\times} \right\}, \\
 & N' = \left\{ \begin{pmatrix} x & \epsilon y \\ y & x \end{pmatrix}, \begin{pmatrix} x & - \epsilon y \\ y & - x \end{pmatrix}  : (x, y) \neq (0, 0), x, y \in \f_\ell \right\}.
\end{align*}

\begin{lem}
\label{doublecosetoperatorN}
  The double coset operator $NN': \Z[G/N] \rightarrow \Z[G/N']$ coincides with the map $\psi^+:\Z[(\Pr^1(\f_\ell) \times \Pr^1(\f_\ell) - \Delta)/\sim] \rightarrow \Z[\HH_\ell]$ in~\eqref{geodesic} and is
  hence a $\z[G]$-module homomorphism.
\end{lem}
\begin{proof}
Since
\begin{equation*}
  N \cap N' = \left\{ \begin{pmatrix} \alpha & 0 \\ 0 & \pm \alpha
\end{pmatrix} : \alpha \in \f_\ell^\times \right\} \cup \left\{  \begin{pmatrix} 0 & \pm \epsilon \alpha \\ \alpha &
0 \end{pmatrix} : \alpha \in \f_\ell^\times \right\},
\end{equation*}
we have from Lemma~\ref{doublecoset} that
\begin{equation*}
  N N' = \cup_{\alpha \in \f_\ell^\times/\left\{ \pm 1 \right\}} \begin{pmatrix} \alpha & 0 \\ 0 & 1 \end{pmatrix}
  N'.
\end{equation*}
The $\z[G]$-module homomorphism from $\z[G/N] \rightarrow \z[G/N']$
induced by $NN'$ from \eqref{doublecoset-defn} is then seen to
be the map $\psi^+$.
\end{proof}

\begin{lem}
\label{doublecosetoperatorC}
The double coset operator $C \begin{pmatrix} 1 & s \\ 0 & 1 \end{pmatrix} C':\Z[G/C] \rightarrow \Z[G/C']$ coincides with the map $H_s:\Z[\Pr^1(\f_\ell) \times \Pr^1(\f_\ell) - \Delta] \rightarrow \Z[\C_\ell]$ in~\eqref{path} and is hence a $\z[G]$-module homomorphism.
\end{lem}
\begin{proof}
For $g = \begin{pmatrix} 1 & s \\ 0 & 1 \end{pmatrix}$, we have that
\begin{equation*}
  C \cap g C' g^{-1} = \left\{ \begin{pmatrix} \alpha & 0 \\ 0 & \alpha
\end{pmatrix} : \alpha \in \f_\ell^\times \right\}.
\end{equation*}
Thus, from Lemma~\ref{doublecoset}, we have that
\begin{equation*}
  C \begin{pmatrix} 1 & s \\ 0 & 1 \end{pmatrix} C' = \cup_{\alpha \in \f_\ell^\times} \begin{pmatrix} \alpha & \alpha s \\ 0 & 1 \end{pmatrix}
  C'.
\end{equation*}
The $\z[G]$-module homomorphism from $\z[G/C] \rightarrow \z[G/C']$
induced by $CC'$ from \eqref{doublecoset-defn} is then seen to
be the map $H_s$.
\end{proof}

\section{Normalizer of Cartan subgroup case}

\label{sectionN}

In this section, we explain and give a detailed proof of Merel's conjecture for normalizers of Cartan subgroups using methods in \cite{zagier}. In this situation, the conjectural explicit intertwining operator is given by a single double coset operator.

Define $\gamma_{\left\{ 0,\infty \right\}} := \f_\ell^\times
\sqrt{\epsilon} \subseteq \HH_\ell$, which can be thought of as the
geodesic in $\HH_\ell$ between $0$ and $\infty$. Given an unordered
pair $\left\{ a,b \right\}$, there is a $g \in G$ such that $\left\{
a , b \right\} = \left\{ g(0), g(\infty) \right\}$, which is unique
up to multiplication on the left by $N$. Thus, we may define
\begin{equation}
\label{geodesic}
  \gamma_{\left\{a, b\right\}} :=
g(\gamma_{\left\{0,\infty\right\}}),
\end{equation}
 which can be thought of as the
geodesic in $\HH_\ell$ between $a$ and $b$.
\begin{lem}
\label{g-matrix}
A choice for the element $g$ above is given by $$\begin{pmatrix}
 b & a \\
 1 & 1
\end{pmatrix}.$$
\end{lem}
\begin{proof}
The point at infinity $\infty$ is denoted by $(1,0)^{t}$ and the point 0 by
$(0,1)^{t}$. We require a matrix $g$ such that $g \cdot 0 =
(a,1)^{t}$ and $g \cdot \infty = (b,1)^{t}$, which is given by the above matrix.
\end{proof}

The finite field $\f_{\ell^2}$ is a vector space over $\f_\ell$ of
dimension $2$. The basis $\left\{ 1, \sqrt{\epsilon} \right\}$ gives
us an identification $\f_{\ell^2} \cong \f_\ell + \sqrt{\epsilon}
\f_\ell$. Thus, for every $z \in \f_{\ell^2}$, we can write $z = x+
\sqrt{\epsilon} y$ for some $x, y \in \f_\ell$.

\begin{lem}
\label{geodesicircle}
The quadratic equation 
\begin{equation}
\label{equN} \left( x - \frac{a+b}{2} \right)^2 - \epsilon y^2 =
\left(\frac{b - a}{2}\right)^2.
\end{equation}
gives the geodesic $\gamma_{\left\{a,b\right\}}$ with coordinates(see Figure~\ref{m2})
\begin{align*}
& x =\frac{a- \epsilon \lambda^2 b}{1-\epsilon \lambda^2},\\
& y= \lambda \left( \frac{a-b}{1-\epsilon \lambda^2} \right).
\end{align*}
\end{lem}

\begin{figure}
\centering
\begin{tikzpicture}[domain=-4:4]
    \draw[->] (-4,0 ) -- (4,0) node[right] {$\f_\ell$}; 
    \draw[->] (0,-3.75) -- (0,3.75) node[above] {$\f_\ell \sqrt\epsilon$};
    \draw(1,0)  circle (2cm) node [black,below left]{$O$};
    \coordinate (O) at (1,0);
    \def\radius{2cm}
    \draw (O) circle[radius=\radius];
    \fill (O) circle[radius=2pt] node[below left] {O};
    \coordinate (a) at (-1,0);
    \fill (a) circle[radius=2pt] node[below left] {$a$};
    \coordinate (b) at (3,0);
    \fill (b) circle[radius=2pt] node[below right] {$b$};
    \draw(1,2) node[above]{$\gamma_{\left\{a,b\right\}}$};
\end{tikzpicture}
\caption{ The geodesic $\gamma_{\left\{a,b\right\}}$ in $\HH_\ell$}
\label{m2}
\end{figure}

\begin{proof}
Writing $g(\lambda\sqrt\epsilon,1)^{t}$ as a fraction and then rationalizing it, we obtain:
\begin{align*}
& \frac{b\lambda \sqrt\epsilon+a}{\lambda \sqrt\epsilon+1} 
%=& \frac{b\lambda\sqrt\epsilon + a}{\lambda\sqrt\epsilon+1}.\frac{\lambda\sqrt\epsilon - 1}{\lambda \sqrt\epsilon - 1}\\
%=& \frac{b\lambda^2 \epsilon + a\lambda \sqrt\epsilon - b\lambda\sqrt\epsilon - a}{\epsilon \lambda^2 - 1}\\
%=& \frac{b\lambda^2 \epsilon - a}{\epsilon \lambda^2-1}+\sqrt\epsilon \frac{\lambda(a-b)}{\epsilon \lambda^2 -1} \\
= \frac{a-b\lambda^2 \epsilon}{1-\epsilon \lambda^2}+\sqrt\epsilon \frac{\lambda(b-a)}{1-\epsilon \lambda^2}.
\end{align*}
Therefore, as $\gamma_{\left\{a,b\right\}} \in \f_{\ell^2}^{\times}$ we conclude $x,y$ from the above expression are given by (see Figure~\ref{m2})
\begin{align*}
& x =\frac{a- \epsilon \lambda^2 b}{1-\epsilon \lambda^2},\\
& y= \lambda \left( \frac{a-b}{1-\epsilon \lambda^2} \right).
\end{align*}
\end{proof}

\subsection{Coordinates for $G/N$ and $G/N'$}~\\

We need a more convenient coordinate to represent elements in (a certain subset of) $(\Pr^1(\f_\ell) \times \Pr^1(\f_\ell) - \Delta)/\sim$ and
$\HH_\ell$, where $(\Pr^1(\f_\ell) \times \Pr^1(\f_\ell) -\Delta)/\sim$ is in bijection with $G/N$, and $\HH_\ell$ is in bijection with $G/N'$.
\begin{lem}
\label{bijectN} Let $A_+ = (\f_\ell \times \f_\ell - \Delta)/\sim$ and
$B_+ = \left\{(t, n): t^2 - 4n \neq 0 \text{ is a square in } \f_\ell
\right\}$. Then there is a bijection between the sets $A_+$ and $B_+$
given by
\begin{align*}
\left\{a, b \right\}  \mapsto (a+b,ab).
\end{align*}
\end{lem}
\begin{proof}
The inverse map is given by $(t,n) \mapsto \left\{ a,b \right\}$, where $\left\{ a, b \right\}$ is the set of roots in $\f_\ell$ of the polynomial $x^2 - t x + n$.

%Let $\kappa : A_+ \rightarrow B_+$ be the map such that $\kappa$ sends each unordered pair $\left\{a, b \right\}$ to $(t, n)$. 
%
%Let $\left\{a,b\right\}$ and $\left\{c,d\right\}$ be two distinct elements of $A$ such that $\kappa(\left\{a,b\right\}) = \kappa(\left\{c,d\right\})$. Then, $(a+b,ab)=(c+d,cd)$ which implies that $(x-a)(x-b)= (x-c)(x-d)$. Furthermore, unique factorization of polynomial ring $\f_\ell[x]$ implies that either $a=c, b=d$ or $a=d, b=c$. Therefore $\left\{a,b\right\} = \left\{c,d\right\}$, which is a contradiction. Thus, $\kappa$ is an injective map.
%
%Let $(t_1,n_1) \in B_+ =\left\{x^2-t_1x+n_1, t_1^2-4n_1 \text{ is  a square in } \f_\ell\right\}$. This equation has two solutions $\frac{t_1+m}{2}$ and $\frac{t_1-m}{2}$ in $\f_\ell$, since $\Delta = t_1^2 - 4n_1=m^2$ is a square in $\f_\ell$. Then, there exist $c, d \in \f_\ell$ such that $\frac{t_1+m}{2} = c$ and $\frac{t_1-m}{2}=d$. Therefore, $n_1 = cd$, and $t_1 = c+d$, hence $(t_1,n_1)=(c+d,cd)$, which proves the surjectivity of $\kappa$.
\end{proof}

\begin{lem}
\label{bijectNp} Let $A'_+ = \HH_\ell$ and $B'_+ = \left\{ (T,N) : T^2 -
4 N \text{ is a non-square in } \f_\ell \right\}$. Then there is a
bijection between the sets $A'_+$ and $B'_+$ given by
\begin{align*}
\left\{z, \bar z\right\} \mapsto (z + \bar z, z\bar z).
\end{align*}

\end{lem}
\begin{proof}
The inverse map is given by $(T,N) \mapsto \left\{ z, \bar z \right\}$, where $\left\{ z, \bar z \right\}$ is the set of roots in $\HH_\ell$ of the polynomial $x^2 - T x + N$.

%Let $A'_+ = \left\{\left\{z, \bar z \right\} : z \in \C_{\ell} \right\}$ and $K= \left\{(T, N): x^2 -Tx+N: \Delta \neq 0 \right\}$. Let $\kappa' : A'_+ \rightarrow K$ be the map such that $\kappa'$ sends each unordered pair $\left\{z, \bar z \right\}$ to $(T, N)$. 
%
%Let $\left\{z_1,\bar z_1\right\}$ and $\left\{z_2,\bar z_2\right\}$ be two distinct elements of $\HH_\ell$ such that $\kappa'(\left\{z_1,\bar z_1\right\}) = \kappa'(\left\{z_2,\bar z_2\right\})$. Then, $(z_1+\bar z_1,z_1\bar z_1)=(z_2+\bar z_2,z_2 \bar z_2)$ implies that $(2x_1, x_1^2 - \epsilon y_1^2) = (2x_2,x_2^2 - \epsilon y_2^2)$ where $z_1 =x_1 + \sqrt \epsilon y_1, \bar z_1 =x_1-\sqrt\epsilon y_1, z_2=x_2+\sqrt \epsilon y_2$ and $\bar z_2=x_2-\sqrt \epsilon y_2$, then $x_1=x_2$ and $y_1=\pm y_2$. Therefore, we have either $z_1=z_2, \bar z_1=\bar z_2$ or $z_1=\bar z_2, \bar z_1=z_2$. Hence, $\left\{z_1,\bar z_1\right\} = \left\{z_2,\bar z_2 \right\}$, which is a contradiction. Thus, $\kappa'$ is an injective map.
%
%Let $(T_1,N_1) \in B'_+=\left\{x^2-T_1x+N_1, T_1^2-4N_1 \text{ is  a non-square in } \f_\ell\right\}$. This equation has two solutions $\frac{T_1+\sqrt\Delta}{2}$ and $\frac{T_1-\sqrt\Delta}{2}$ in $\HH_\ell$, since $\Delta = T_1^2 - 4N_1$ is a non-square in $\f_\ell$. Then, there exist $z', \bar z' \in \HH_\ell=\C_\ell/\sim$ such that $\frac{T_1+\sqrt\Delta}{2} = z'$ and $\frac{T_1-\sqrt\Delta}{2}=\bar z'$. Therefore, $N_1 = z'\bar z'$, and $T_1 = z'+\bar z'$, hence $(T_1,N_1)=(z'+\bar z',z'\bar z')$, which proves the surjectivity of $\kappa'$. 
\end{proof}

\begin{lem}
\label{lemm}
Let 
\begin{align*}
& B_+ =\left\{(t,n):t^2 - 4n \neq 0 \text{ is a square in } \f_\ell \right\},\\ 
& S_+ = \left\{(t, m):m \text{ is a square in } \f_\ell\right\}.
\end{align*}
Then there is a bijection between the sets $B_+$ and $S_+$ given by:
\begin{align*}
& (t,n) \mapsto (t, m),
\end{align*}
where $m= t^2 - 4n$.
\end{lem}

\begin{proof}
The inverse map is given by $(t,m) \mapsto (t, \frac{t^2 - m}{4})$.
\end{proof}

\begin{lem}
\label{lemM}
Let 
\begin{align*} 
& B'_+ = \left\{(T,N):T^2 - 4N \text{ is a non-square in }\f_\ell \right\}, \\
& S'_+ = \left\{(T,M): M \text{ is a non-square in }\f_\ell \right\}. \end{align*} Then there is a bijection between the sets $B'_+$ and $S'_+$ given by:
\begin{align*}
& (T,N) \mapsto (T,M),
\end{align*}
where $M = T^2 - 4N$.
\end{lem}
\begin{proof}
The inverse map is given by $(T,M) \mapsto (T, \frac{T^2 - M}{4})$.
\end{proof}

\subsection{Proof of Theorem~\ref{main-N}}~\\

By Lemma~\ref{doublecosetoperatorN}, $\psi^+$ is a $\q[G]$-module homomorphism. To prove Theorem~\ref{main-N}, it suffices to prove that the
restriction
\begin{equation}
\label{equationmain}
  \psi^+_{\mid_{\q[A_+]}} : \q[A_+] \rightarrow \q[\HH_\ell],
\end{equation}
is an isomorphism of $\q$-vector spaces.

Using the bijections given by Lemmas~\ref{bijectN}-\ref{lemM},
to prove ~\eqref{equationmain} is equivalent to proving that
\begin{equation*}
 \psi^+: \q[S_+] \rightarrow \q[S'_+],
\end{equation*}
is an isomorphism of $\q$-vector spaces, where $\psi^{+}$ is the same map as $\psi^{+}_{\mid_{\q[A]}}$ under the identifications given by two bijections $A_+ \leftrightarrow S_+$ and $\HH_\ell \leftrightarrow S'_+$.

Recall the equation giving the geodesic between $a$ and $b$ is
\begin{equation*}
\left( x - \frac{a+b}{2} \right)^2 - \epsilon y^2 = \left(\frac{b-a}{2}\right)^2,
\end{equation*}
by Lemma~\ref{geodesicircle}. This equation becomes
\begin{align*}
 & \left( x - \frac{a+b}{2} \right)^2 - \epsilon y^2 = \left(\frac{b-a}{2}\right)^2 \\
% \iff & \left( x - \frac{t}{2} \right)^2 - \epsilon y^2 = \left( \frac{l}{2} \right)^2 \\
% \iff &  4 \left( x - \frac{t}{2} \right)^2 - 4 \epsilon y^2 = l^2 \\
% \iff &   (2x - t)^2 - 4\epsilon y^2 = l^2 \\
% \iff &  (T-t)^2 - 4\epsilon y^2 = l^2  \\
 \iff & (T-t)^2 = l^2 + 4\epsilon y^2 = m+M,
\end{align*}
in the new coordinates from Lemmas \ref{lemm} and \ref{lemM}. Here,  $m$ and $M$ satisfy $(\frac{m}{\ell})=1$ and $(\frac{M}{\ell})= -1$, where $(\frac{\cdot}{\ell})$ is the Legendre symbol modulo $\ell$.

Hence, the matrix of $\psi^+_{\mid \q[S_+]}$ with respect to the basis
$S_+$ is given by
\begin{equation}
  a_{(T,M),(t,m)} =
  \begin{cases}
    1 & \text{ if } (T-t)^2 \equiv m+M \pod{\ell} ,\\
    0 & \text{otherwise}.
  \end{cases}
\end{equation}

Thus, the above matrix is an $\frac{\ell-1}{2} \times
\frac{\ell-1}{2}$ matrix $D_{m,M}$, with entries being the $\ell
\times \ell$ matrices given by
\begin{align*}
  (D_{m,M})_{t,T} :=
  \begin{cases}
    1 & \text{ if } (T-t)^2 \equiv m+M \pod{\ell}, \\
    0 & \text{otherwise}.
  \end{cases}
\end{align*}

Let $D$ be the matrix obtained from the $\ell \times \ell$ identity matrix by permuting its columns according to the cycle $(1\,2\,3...\,\ell)$.

\begin{defn}
A circulant matrix is a matrix of the form
\begin{equation*}
\begin{pmatrix}
a_0 & a_1 & a_2 & ... & a_{r-1} \\ a_{r-1} & a_0 & a_1 & ... &
a_{r-2} \\ \vdots & & & & \vdots \\ a_{1} & a_{2} & ... & a_{r-1} & a_0
\end{pmatrix},
\end{equation*}
that is, a matrix whose $i$-th row is obtained from the $(i-1)$-th
row by cyclically shifting the entries one position to the right.
\end{defn}
\begin{lem}
$D_{m,M} = \sum_{x^2 \equiv m+M (\ell)} D^x$
\end{lem}
\begin{proof}
If $m + M$ is not a square in $\f_\ell$, therefore $D_{m,M}$ is a
zero matrix due to 0 entries, so $D_{m,M}= \sum _{x^2 \equiv m + M
(\ell)} D^x = 0$.

If $m + M = x^2$ is a square in $\f_\ell$. Then $T-t = \pm x$
and
\begin{equation*}
  (D_{m,M})_{t,T} = \begin{cases}
    1 & T = t \pm x ,\\
    0 & \text{otherwise}.
  \end{cases}
\end{equation*}
In this case, $D_{m,M}$ coincides with $\sum_{x^2 \equiv m+M \pod\ell} D^x$.
\end{proof}

Let $\zeta$ be an $\ell$-th root of unity. The matrix $D_{m,M}$ has
entries in $\q[D]$, but we can replace the matrix $D$ by an element
in the cyclotomic field $\q(\zeta)$ in the following manner: the minimal polynomial of $D$ over $\q$ is given by $m(x) = x^{\ell-1}+\cdots+x+1$, so we have that
\begin{align*}
& \q[D] \cong \frac{\q[x]}{(m(x))} \cong \q[\zeta] \cong \q(\zeta),
\end{align*}
where $\zeta$ is a primitive $\ell$-th root of unity.
\begin{lem}
\label{ramified}
Let $\mathfrak{L}$ be a prime of $\q(\zeta)$ above $\ell$. Then
$\zeta \equiv 1 \pod{\mathfrak{L}}$.
\end{lem}
\begin{proof}
\cite[lemma 10.1]{Neukirch}.
\end{proof}

From the above discussion, we see that $D_{m,M} = \sum_{x^2 \equiv m +M(\ell)} 1$ (after reduction modulo $\mathfrak{L}$). We label $m,M$ as $m=g^{2i}$ for $0 \le i \le r-1$ and $M = \epsilon g^{2j}$ for $0 \le j \le r-1$, where $r=\frac{\ell-1}{2}$ and $g$ is a primitive root modulo $\ell$. This gives us a new matrix denoted by $D_{i,j}$:

\begin{equation}
\label{circulant}
  D_{i,j} = \sum_{x^2 \equiv g^{2i} + \epsilon g^{2j} \pod{\ell}  } 1.
\end{equation}

\begin{prop}
\label{determinantofcirculant}
The determinant of a circulant matrix is given by
\begin{equation}
\label{circulantM}
  \prod_{k=0}^{r-1} (a_0 + a_1 \omega_k + a_2 \omega_k^2 + ... + a_{r-1} \omega_k^{r-1}) = \prod_{k=0} ^{r-1} \left(\sum_{j=0} ^{r-1} a_j \omega_k ^{j}\right),
\end{equation}
where $\omega_k = e^{\frac{2\pi i k}{r}} = \omega^k, r \ge 1$ and
 $\omega = e^{\frac{2\pi i}{r}}$.
\end{prop}

\begin{proof}
Suppose
\begin{equation*}
A = \begin{pmatrix} a_0 & a_1 & a_2 & ... & a_{r-1} \\ a_{r-1} & a_0
& a_1 & ... & a_{r-2} \\ \vdots & \vdots & & & \vdots \\ a_{1} & a_{2} & a_3 & ... &
a_0
\end{pmatrix},
\end{equation*}
is a circulant matrix. Let $\omega_{k} = e^{2\pi i k/r}$ for $0 \le k \le
r-1$. Now, consider the row vector $(1, \omega_k,
\omega_k ^2, ..., \omega_k ^{r-1})$, whose transpose we denote by
$\gamma_{k} \in \mathbb{C}^r$, and let $\sigma_{k} = a_0 + a_1
\omega_{k} + a_2 \omega_k ^2 + ... + a_{r-1} \omega_k ^{r-1}$.
Then we get that
\begin{equation*}
\begin{pmatrix} a_0 & a_1 & ... & a_{r-1} \\ a_{r-1} & a_0 & ... & a_{r-2} \\ \vdots \\ a_1 & a_2 & ... & a_0 \end{pmatrix} \begin{pmatrix} 1 \\ \omega_k \\ \vdots \\ \omega_k ^{r-1} \end{pmatrix} = \sigma_k \begin{pmatrix} 1 \\ \omega_k \\ \vdots \\ \omega_k ^{r-1} \end{pmatrix},
\end{equation*}
which implies that $\sigma_k$ is an eigenvalue of $A$ with
eigenvector $\gamma_k$. Furthermore, the set
$\left\{\gamma_0, \gamma_1, ..., \gamma_{r-1} \right\}$ is a
linearly independent set in $\mathbb{C}^{r}$, since the eigenvalues $\sigma_{k}$ are distinct. Therefore, a diagonal
matrix with the corresponding eigenvalues is conjugate to $A$, and hence
the determinant of $A$ is given by $\text{det}(A) =
\prod_{k=0} ^{r-1} \sigma_k$.
\end{proof}
For later reference, we call each factor in the above formula \textit{an eigenvalue} for $k$. We also let $r = \frac{\ell - 1}{2}$ in this section.
\begin{lem}
\label{equationmatrix} The matrix $D_{i,j}$ is an $r \times r$
circulant matrix.
\end{lem}
\begin{proof}

This follows because
\begin{equation*}
 D_{i,j} = \sum_{x^2 \equiv g^{2i} + \epsilon g^{2j}} 1 \equiv \sum_{x^2 \equiv g^2(g^{2(i-1)} +\epsilon g^{2(j-1)})} 1 = D_{i-1, j-1},
\end{equation*}
where the indices are taken modulo $\ell$.
\end{proof}

Remark that $D_{0,j} = a_j$ is equal to the number of solutions of $x^2 \equiv 1 + \epsilon g^{2j} \pod \ell$. To show that $D_{i,j}$ has non-zero determinant, it suffices to show that $D_{i,j}$ has non-zero determinant modulo $\ell$ in ~\eqref{circulantM}.

Using the above formula for the determinant of a circulant matrix,
it suffices to show in $\z[\omega]$ that we have
\begin{equation}
\label{nonzerow} a_0 + a_1 \omega_k +a_2 \omega_k^2 + ... + a_{r-1}
\omega_k^{r-1} \not\equiv 0 (\vartheta)
\end{equation}
for every $0 \le k \le r-1$, where $\vartheta$ is any prime above $\ell$ in $\z[\omega]$, $\omega_k = \omega^k$ and $\omega = e^{2\pi i/r}$.

\begin{lem}
  Let $\vartheta$ be a prime above $\ell$ in $\z[\omega]$ where $\omega = e^{2\pi i/r}$. Then
  $\omega \equiv g^2 (\vartheta)$, where $g$ is a primitive root
  modulo $\ell$.
\end{lem}
\begin{proof}
Let $\mathcal{O} = \z[\omega]$ be the maximal order of $\q(\omega)$. The residue
field of $\vartheta$ is $\mathcal{O}/\vartheta \cong \f_\ell$.
Furthermore, since the polynomial $x^r - 1$ splits in
$\mathcal{O}/\vartheta[x] \cong \f_\ell[x]$ with distinct roots
$\omega_1 = \omega, \omega_2 = \omega^2, \ldots, \omega_r = \omega^r
= 1$, we have that every root of $x^r-1$ in $\f_\ell$ is a power of
$\omega \in \mathcal{O}/\vartheta \cong \f_\ell$. Hence, $\omega
\cong g^2 \pod{\vartheta}$ for some primitive root $g$ modulo
$\ell$.
\end{proof}

%The following is well-known but we state it because it is used repeatedly.
%\begin{lem}
%\label{orthogonal}
%Let $k \in \mathbb{Z}$, and $\ell$ be a prime. Then
%\begin{equation*}
%  \sum_{x=1} ^{\ell-1} x^k \equiv
%  \begin{cases}
%    0 \pod{\ell} & \text{ if } k \not\equiv 0 \pod{\ell-1} \\
%    -1 \pod{\ell} & \text{ if } k \equiv 0 \pod{\ell -1}.
%  \end{cases}
%\end{equation*}
%\end{lem}

By the above lemma, to show \eqref{nonzerow}, it suffices to show
\begin{lem}
\begin{equation*}
 \sum _{j=0} ^{r-1} a_j (g^{2k})^j \not\equiv 0 (\ell),
\end{equation*}
for every $0 \le j, k \le r-1$.
\end{lem}
\begin{proof}
The above sum can be calculated as:
\begin{align}
  & \sum_{j=0} ^{r-1} D_{0,j} (g^{2k})^j \equiv \sum _{j=0} ^{r-1} a_j (g^{2k})^j 
  \equiv \sum_{j=0}^{r-1} \left( \sum_{x^2 \equiv 1 + \epsilon g^{2j} \pod \ell} 1 \right) (g^{2j})^k \notag\\
  \equiv  & \sum_{\substack{j=0,...,r-1\\ x=0,...,\ell-1 \\ x^2 \equiv 1 + \epsilon g^{2j} \pod \ell}} (g^{2j})^k 
  \equiv  2 \sum_{\substack{x=0,...,\ell-1\\ y = 1,\ldots,\ell-1 \\ x^2 \equiv 1 + \epsilon y^{2}\pod \ell}} \left(\frac{x^2 - 1}{\epsilon} \right)^k \label{nonzero}.
\end{align}
Now, we need to show that \eqref{nonzero} is non-zero modulo $\ell$ for every $0 \le k \le r-1$.

We can rewrite $(\frac{x^2 - 1}{\epsilon})$ as $y^2$ since $y=g^j$ for $0 \le j \le r-1$. The conic $x^2 \equiv 1 + \epsilon y^2 \pod \ell$ is parametrized by $x= \frac{a - \epsilon \lambda^2 b}{1 - \epsilon\lambda^2}, y= \frac{\lambda(a-b)}{1- \epsilon \lambda^2}$ from \eqref{equN}. Here, we compute $D_{i,j}$ for $i=0$ which corresponds to $m=1=(a-b)^2$.
Thus, we can rewrite \eqref{nonzero} as
\begin{align*}
  \sum_{\lambda =1} ^{\ell -1} \left(\frac{\lambda}{1- \epsilon \lambda^2} \right) ^{2k} 
  \equiv \sum_{\lambda = 1} ^{\ell -1} (\lambda^{-1} - \epsilon \lambda)^{-2k} 
  \equiv \sum_{\lambda =1} ^{\ell -1} (\lambda^{-1} - \epsilon \lambda)^{2k'},
\end{align*}
where $-2k' \equiv 2k \pod{\ell -1} $ and $0 \le 2k' \le \ell - 2$, hence $0 \le k' \le \frac{\ell -3}{2}$. Here, we have to consider two cases.

The first case is $k \ge 1$.  The sum of all terms except the constant terms will be zero modulo $\ell$. Therefore, we just have to compute the sum of the constant terms which is
\begin{align}
 \label{equationsum} \sum_{\lambda = 1} ^{\ell - 1} \frac{(2k')!}{k'! k'!} (-1)^{k'} \epsilon ^{k'} \equiv  \frac{(2k')!}{k'! k'!} (-1)^{k'+1} \epsilon^{k'} \pod{\ell},
\end{align}
which is non-zero modulo $\ell$, since $2k' < \ell$ for all values of $k'$, hence none of the terms of ~\eqref{equationsum} is a multiple of $\ell$, therefore it is non-zero modulo $\ell$.

The second case is when $k=0$, then the sum in \eqref{equationsum} becomes $\sum_{\lambda=1} ^{\ell -1} 1 \equiv -1 \not\equiv 0 \pod\ell$.
\end{proof}
This concludes the proof of Theorem ~\ref{main-N}.

\section{Cartan subgroup case}

In this section, we generalize Merel's conjecture to Cartan subgroups and give a proof by generalizing the methods in Section~\ref{sectionN}. A new feature is that the conjectural explicit intertwining operator is now a linear combination of double coset operators (rather than a single double coset operator) whose coefficients we are able to make explicit.

Define $\gamma_{(0, \infty)} := \f_{\ell} ^{\times} \sqrt{\epsilon}
\subseteq \C_{\ell}$, which can be thought of as a path in
$\C_{\ell}$ from $0$ to $\infty$. Given an ordered pair $(a, b)$,
there is a $g \in G$ such that $(a, b) = (g(0), g(\infty) )$, which
is unique up to multiplication on the left by $C$. Thus, we may
define $\gamma_{(a, b)} := g(\gamma_{(0, \infty)})$, which can be
thought as a path in $\C_{\ell}$ from $a$ to $b$ (see Figure~\ref{m3}).

Here, for each $s=1, \ldots, \ell-1$, we define the linear operator $H_s$ by:
\begin{align}
\label{main-Ceq} 
 H_s : \q[\Pr^1(\f_\ell) \times \Pr^1(\f_\ell) - \Delta] \longrightarrow \q[\C_\ell]&\\
 \notag (a, b)\longmapsto \sum_{x \in \gamma^{s}_{(a,b)}} x.
\end{align}

\begin{defn}
Define $\gamma^s _{(0, \infty)}$ to be $\left\{(\lambda s + \lambda
\sqrt{\epsilon}, 1)^{t}: \lambda \in \f_\ell^{\times}, s\in \f_\ell^{\times} \right\} \subseteq \C_\ell$. This is a path in $\C_\ell$ which is a line with slope $s$.
\end{defn}

By Lemma~\ref{g-matrix}, we know that $g = \begin{pmatrix} b & a \\ 1 & 1 \end{pmatrix}$, hence the path in $\C_\ell$ from $a$ to $b$ can be obtained as
\begin{equation}
\label{path}
 \gamma_{(a,b)}^{s} = g(\gamma_{(0,\infty)}^s) = g(\lambda s + \lambda \sqrt{\epsilon}, 1)^{t} = (bs \lambda + b \lambda \sqrt{\epsilon}+a, \lambda s + \lambda \sqrt{\epsilon}+1)^{t},
\end{equation}
which is represented by an equation defined by the next lemma.

\begin{lem}
\label{pathlem}
The quadratic equation
\begin{equation}
\label{equM} \left(x - \frac{a+b}{2} \right)^2 - \epsilon
\left(y - \frac{s(b-a)}{2\epsilon} \right)^2 =
\frac{(\epsilon - s^2)(a - b)^2 }{4 \epsilon}
\end{equation}
gives the path $\gamma^s_{(a, b)}$ with coordinates (see Figure~\ref{m3})
\begin{align*}
 & x= \frac{(b \lambda s + a)(\lambda s + 1) - b \lambda^2 \epsilon}{(\lambda s + 1)^2 - \lambda ^2 \epsilon},\\
 & y = \frac{\lambda (b-a)}{(\lambda s + 1)^2 - \lambda ^2 \epsilon }.
\end{align*}

\end{lem}
        
\begin{figure}
\centering
\begin{tikzpicture}[domain=-6:6]
    \draw[->] (-5,0 ) -- (6,0) node[right] {$\f_\ell$}; 
    \draw[->] (0,-3.5) -- (0,4) node[above] {$\f_\ell \sqrt\epsilon$};
    \draw(3,1)  circle (2cm) node [black,below left]{$O$};
    \coordinate (O) at (3,1);
    \def\radius{2cm}
    \draw (O) circle[radius=\radius];
    \fill (O) circle[radius=2pt] node[below left] {O};
    \coordinate (a) at (1.3,0);
    \fill (a) circle[radius=2pt] node[below left] {$a$};
    \coordinate (b) at (4.7,0);
    \fill (b) circle[radius=2pt] node[below right] {$b$};
    \draw(3,3) node[above]{$\gamma_{\left(a,b\right)}$};
\end{tikzpicture}
\caption{The path $\gamma_{\left(a,b\right)}$ in $\C_\ell$}
\label{m3}
\end{figure}   
        
\begin{proof}
Writing $g (\lambda s + \lambda \sqrt{\epsilon}, 1)^{t}$ as a fraction and then rationalizing it, we obtain:
\begin{align*}
% \frac{b \lambda s + b \lambda \sqrt{\epsilon} + a}{\lambda s + \lambda \sqrt{\epsilon} + 1} 
\frac{(b \lambda s + a) + b \lambda \sqrt{\epsilon}}{(\lambda s+ 1) + \lambda \sqrt{\epsilon}} 
%  = &  \frac{b \lambda s + b \lambda \sqrt{\epsilon} + a}{(\lambda s+ 1) - \lambda \sqrt{\epsilon}}. \frac{(\lambda s + 1) + \lambda \sqrt{\epsilon}}{(\lambda s + 1) - \lambda \sqrt{\epsilon}} \\
%  = & \frac{(b \lambda s + b \lambda \sqrt{\epsilon} + a)(\lambda s + 1 - \lambda \sqrt{\epsilon})}{(\lambda s + 1)^2 - \lambda ^2 \epsilon}  \\
%  = & \frac{b \lambda ^2 s^2 + b \lambda ^2 s \sqrt{\epsilon} + a \lambda s + b \lambda s + b \lambda \sqrt{\epsilon} + a - b \lambda ^2 s \sqrt{\epsilon} - b \lambda ^2 \epsilon - a \lambda \sqrt{\epsilon}}{(\lambda s + 1)^2 - \lambda ^2 \epsilon}  \\
%  = & \frac{b \lambda s (\lambda s + 1) + a(\lambda s + 1) - b \lambda ^2 \epsilon + \lambda \sqrt{\epsilon}(b -a)}{(\lambda s + 1)^2 - \lambda ^2 \epsilon}  \\
  =  \frac{(b \lambda s + a)(\lambda s + 1) - b \lambda ^2 \epsilon}{(\lambda s + 1)^2 - \lambda ^2 \epsilon} + \sqrt{\epsilon}\frac{\lambda (b -a)}{(\lambda s + 1)^2 - \lambda ^2 \epsilon}.
\end{align*}
Therefore, as $\gamma_{(a, b)}^s \in \f_{\ell^2}^{\times}$ we conclude $x,
y$ from the above expression are given by (see Figure~\ref{m3})
\begin{align*}
  &x = \frac{(b \lambda s + a)(\lambda s + 1) - b \lambda ^2 \epsilon}{(\lambda s + 1)^2 - \lambda ^2 \epsilon}, \\
  &y = \frac{\lambda (b -a)}{(\lambda s + 1)^2 - \lambda ^2 \epsilon}.
\end{align*}
\end{proof}

\subsection{Coordinates for $G/C$ and $G/C'$}~\\

We need more convenient coordinates to represent elements in (a certain
subset of) $\Pr^1(\f_\ell) \times \Pr^1(\f_\ell) - \Delta$ and
$\C_\ell$, where $\Pr^1(\f_\ell) \times \Pr^1(\f_\ell) - \Delta$ is in bijection with $G/C$ and $\C_\ell$ is in bijection with $G/C'$.

\begin{lem}
\label{lemt'}
Let $A = \f_\ell \times \f_\ell - \Delta$ and $S= \left\{(t, t'): t'
\neq 0 \right\}$. Then there
is a bijection between the sets $A$ and $S$ given by:
\begin{align*}
 & (a, b) \mapsto (a+b, a-b).
\end{align*}
\end{lem}

\begin{proof}
The inverse map is given by $a= \frac{x+y}{2}$ and $b=\frac{x-y}{2}$.
\end{proof}

\begin{lem}
\label{lemT'}
Let $A'= \C_\ell $ and $S'= \left\{(T, T'): T' \neq 0 \right\}$. Then there is a bijection between
the sets $A'$ and $S'$ given by:
\begin{align*}
 & (z, \bar z) \mapsto (z+\bar z, z - \bar z).
\end{align*}
\end{lem}

\begin{proof}
The inverse map is given by $z= \frac{z+ \bar z}{2} + \sqrt\epsilon \frac{z - \bar z}{2}$ and $\bar z = \frac{z + \bar z}{2} -\sqrt\epsilon \frac{z - \bar z}{2}$.
\end{proof}

\subsection{Proof of Theorem~\ref{main-C}}~\\

By Lemma~\ref{doublecosetoperatorC}, $\psi$ is a $\q[G]$-module homomorphism. To prove Theorem~\ref{main-C}, it suffices to prove that the restriction
\begin{equation}
\label{theorem2}
\psi_{\mid \q[A]} : \q[A] \rightarrow \q[\C_\ell],
\end{equation}
is an isomorphism of $\q$-vector spaces.

Using the bijections given by Lemma~\ref{lemt'} and Lemma~\ref{lemT'}, to prove \eqref{theorem2} is equivalent to proving that
\begin{equation*}
\psi: \q[S] \rightarrow \q[S'],
\end{equation*}
is an isomorphism of $\q$-vector spaces, where $\psi$ is the same map as $\psi_{\mid{\q[A]}}$ under identifications given by two bijections $A \leftrightarrow S$ and $\C_\ell \leftrightarrow S'$.

Recall the equation giving the path $\gamma^s_{(a,b)}$ from $a$ to $b$ is
\begin{equation*}
\left(x - \frac{a+b}{2} \right)^2 - \epsilon \left(y - \frac{s(b-a)}{2\epsilon} \right)^2 = \frac{(\epsilon - s^2)(a-b)^2}{4 \epsilon},
\end{equation*}
by Lemma \ref{pathlem}. By the bijections $P \leftrightarrow E$ and $P' \leftrightarrow E'$, this equation becomes
\begin{align}
\notag & \left(x - \frac{a+b}{2} \right)^2 - \epsilon \left(y - \frac{s(b-a)}{2\epsilon} \right)^2 = \frac{(\epsilon - s^2)(a-b)^2}{4 \epsilon} \\
%\notag \iff & \frac{1}{4} (2x - (a+b))^2 = \frac{(\epsilon - s^2)(a-b)^2}{4\epsilon}+ \frac{\epsilon (2\epsilon y - s(b-a))^2}{4 \epsilon^2} \\
%\notag \iff & (2x - (a+b))^2 = \frac{(\epsilon - s^2)(a-b)^2 + (2\epsilon y - s(b-a))^2}{\epsilon} \\
%\notag \iff & (T - t)^2 = \frac{ \epsilon a^2 - 2\epsilon ab+ \epsilon b^2 + 4 \epsilon^2 y^2 - 4 \epsilon y sb +4 \epsilon ysa}{\epsilon} \\
\label{newquad} \iff & (T-t)^2 = (a-b)^2 + 4\epsilon y^2 + 4sy(b-a),
\end{align}
in the new coordinates from Lemma~\ref{lemt'} and Lemma~\ref{lemT'}.
Hence, the matrix of $H_s$ restricted to $\q[S]$ with respect to the basis $S$ is given by
\begin{equation}
  a_{(t,t'),(T,T')}(s) =
  \begin{cases}
    1 & \text{  if  } (T-t)^2\equiv t'^2 + 4 \epsilon T'^2 + 4sT't' \pod \ell ,\\
    0 & \text{  otherwise}.
  \end{cases}
\end{equation}
The above matrix is an $(\ell-1) \times (\ell-1)$ matrix $X_{t',T'}(s)$, with entries being the $\ell\times \ell$ matrices $(X_{t',T'})_{t,T}$(s) given by
\begin{equation*}
 (X_{t',T'})_{t,T}(s) = \begin{cases}
  1 & \text{  if  } (T-t)^2\equiv t'^2 + 4 \epsilon T'^2 + 4sT't' \pod \ell ,\\
0 & \text{  otherwise}.\end{cases}
\end{equation*}
Let $X$ be the matrix which permutes columns of the $\ell \times \ell$ identity matrix according to the cycle $(1\,2\, 3\cdots \ell)$.
\begin{lem}
$X_{t',T'} (s) = \sum_{v^2 \equiv t'^2 + 4 \epsilon T'^2 + 4sT't' \pod \ell} X^v$.
\end{lem}
\begin{proof}
If $t'^2 + 4 \epsilon T'^2 + 4sT't'$ is not square in $\f_\ell$, then $X_{t,T}(s)$ is a zero matrix due to $0$ entries. Therefore, $X_{t,T}(s) = \sum_{v^2 \equiv t'^2 + 4 \epsilon T'^2 + 4sT't' \pod \ell} X^v = 0$.

If $t'^2 + 4 \epsilon T'^2 + 4sT't'=v^2$ is a square in $\f_\ell$, then $T - t=\pm v$ and
\begin{align*}
(X_{t',T'})_{t,T}(s) = \begin{cases} 1 \quad T = t \pm v ,\\ 0 \quad \text{otherwise}. \end{cases}
\end{align*}
In this case, $X_{t',T'}(s)$ coincides with $\sum_{v^2 \equiv t'^2+4\epsilon T'^2+4sT't' \pod\ell} X^v$.
\end{proof}

Arguing similarly as in the discussion preceeding Lemma~\ref{ramified}, we obtain that $X_{t',T'}(s) = \sum_{v^2 \equiv t'^2 + 4 \epsilon T'^2 + 4sT't' \pod \ell} 1$. We label $t', T'$ as $t' = g^i$ and $T'= g^j$ for $0 \le i, j \le \ell -1$, and $(T-t)^2 = v^2$. This gives us a new matrix denoted by $X_{i,j}(s)$ which is given by
\begin{equation*}
 X_{i,j}(s) =\sum_{v^2 \equiv g^{2i} + 4 \epsilon g^{2j} - 4s g^{i+j}\pod\ell} 1 \pod{\ell}.
\end{equation*}

\begin{lem}
\label{lemma29} The matrix $X_{i,j}(s)$ is a $(\ell - 1) \times (\ell -1)$ circulant matrix.
\end{lem}
\begin{proof}
This follows since
\begin{align*}
& X_{i,j}(s) \equiv \sum_{v^2 \equiv g^{2i}+4\epsilon g^{2j}-4sg^{i+j}\pod \ell} 1 \equiv \sum_{v ^2 \equiv g^2(g^{2(i-1)}+4\epsilon g^{2(j-1)}-4sg^{i+j -2})\pod\ell} 1 \equiv X_{i-1,j-1}(s) \pod\ell,
\end{align*}
where the indices are taken modulo $\ell$. Remark, $X_{0, j}(s)= a_j(s)$ is equal to the number of solutions of $v^2 \equiv 1 + 4 \epsilon g^{2j} - 4sg^j \pod \ell$.
\end{proof}

Let $a_j(s) = X_{0,j}(s)=c_j$, and $\omega = e^{\frac{2 \pi i}{\ell - 1}}$.

\begin{lem}
Let $\vartheta$ be a prime above $\ell$ in $\Z[\omega]$ where $\omega = e^{\frac{2\pi i}{\ell-1}}$. Then $\omega \equiv g \pod \vartheta$, where $g$ is a primitive root modulo $\ell$.
\end{lem}

\begin{proof}
Let $\mathcal{O}= \z[\omega]$ be the maximal order of $\q(\omega)$. The residue field of $\vartheta$ is $\mathcal{O}/\vartheta \cong \f_\ell$ by \cite[Proposition 10.3]{Neukirch}. Furthermore, since the polynomial $x^{\ell -1} - 1$ splits in $\mathcal{O}/\vartheta[x] \cong \f_\ell[x]$ with distinct roots $\omega_1=\omega, \omega_2=\omega^2, \cdots,\omega_{\ell-1}=\omega^{\ell-1} = 1$, we have that every root of $x^{\ell-1}-1$ in $\f_\ell$ is a power of $\omega \in \mathcal{O}/\vartheta \cong \f_\ell$. Hence, $\omega \cong g \pod\vartheta$ for some positive root $g$ modulo $\ell$.
\end{proof}

The eigenvalues of $H_s$ modulo $\vartheta$ can be calculated as
\begin{align}
\label{proof-of-thm-2}
\notag & \sum_{j = 0 } ^{\ell - 2} a_j(s) \omega ^{kj} \equiv_{\vartheta} \sum_{j=0} ^{\ell-2} a_j(s) (g^{k})^j \equiv \sum_{j=0} ^{\ell-2} \left(\sum_{v ^2 \equiv 1 + 4 \epsilon g^{2j} - 4s g^j\pod\ell} 1\right)(g^{k})^j \\
\equiv & \sum_{j=0} ^{\ell-2} \sum_{v ^2 \equiv 1 + 4 \epsilon g^{2j} - 4s g^j\pod\ell} g^{kj}
\equiv \sum_{\lambda =1}^{\ell-1} y(\lambda)^{k}
 \equiv \sum_{\lambda =1}^{\ell-1} \frac{\lambda^k (a-b)^k}{\left((\lambda s + 1)^2 - \lambda^2 \epsilon \right)^k} \pod \ell.
\end{align}
Here, $a_j(s) = X_{0,j}(s)$, which corresponds to $m=1=(a-b)^2$.

We now consider a linear combination $\sum_{s=1} ^{\ell-1} \alpha_s H_s : \q[\Pr^1(\f_\ell) \times \Pr^1(\f_\ell) - \Delta] \rightarrow \q[\C_\ell]$ of the maps $H_s$. Note that a linear combination of circulant matrices is circulant. The eigenvalue of $\sum_{s=1} ^{\ell-1} \alpha_s H_s$ is thus given by $\sum_{j=0} ^{\ell-2} b_j \omega ^{kj}$, where $b_j = \sum_{s=1} ^{\ell-1} \alpha _s a_j
(s)$. Then, we have that
\begin{align*}
 & \sum_{j=0} ^{\ell - 2} b_j \omega ^{kj} = \sum_{j=0} ^{\ell - 2} \left( \sum_{s=1} ^{\ell-1} \alpha_s a_j(s) \right) \omega^{kj} = \sum_{s=1} ^{\ell-1} \alpha_s \sum_{j=0} ^{\ell - 2} a_j(s) \omega^{kj} \\
 \equiv_{\vartheta} & \sum_{s=1} ^{\ell-1} \alpha_s \sum_{\lambda=1}^{\ell-1} y(\lambda, s)^k = \sum_{\lambda=1}^{\ell-1} \sum_{s=1} ^{\ell-1} \alpha_s y(\lambda, s)^k \\
 = & \sum_{s=1} ^{\ell-1} \alpha_s \sum_{\lambda=1}^{\ell-1} \left(\frac{\lambda}{(\lambda s + 1)^2 - \epsilon \lambda ^2} \right)^{k}. \\
\end{align*}

\begin{lem}
\label{lemA}
Let $\alpha_s \equiv 1\pod\ell$ for $s \in \f_\ell^\times$, then the sum
\begin{align}
 \sum_{s=1} ^{\ell-1} \alpha_s \sum_{\lambda=1}^{\ell-1} \left(\frac{\lambda}{(\lambda s + 1)^2 - \epsilon \lambda ^2} \right)^{k},
\end{align}
is non-zero modulo $\ell$ for $k$ even.
\end{lem}
\begin{proof}

In the case $k = 0$, we cannot use a binomial expansion so we  perform a direct computation:
\begin{align}
\label{casezero}
 & \sum_{s=1} ^{\ell-1} \alpha_s \sum_{\lambda=1}^{\ell-1} \left(\frac{\lambda}{(\lambda s + 1)^2 - \epsilon \lambda ^2} \right)^{k} \equiv \sum_{s=1}^{\ell-1} \alpha_s (\ell-1) \notag \\
\equiv & (\ell-1) \sum_{s=1} ^{\ell-1} \alpha_s \equiv (\ell-1)\sum_{s=1} ^{\ell-1} 1 \equiv (\ell-1)(\ell-1) \equiv 1 \pod \ell.
\end{align}
If $k > 0$, then choose $k' \in \mathbb{N}$ such that $k \equiv -k' \pod {\ell-1}$, and $1 \le k' \le \ell -2$. Then
\begin{align*}
 & \sum_{s=1} ^{\ell-1} \alpha_s \sum_{\lambda=1}^{\ell-1} \left(\frac{(\lambda s+ 1)^2 - \epsilon \lambda^2}{\lambda} \right)^{-k} \\
 \equiv & \sum_{s=1} ^{\ell-1} \alpha_s \sum_{\lambda=1}^{\ell-1} \left(\frac{(\lambda s + 1)^2 - \epsilon \lambda^2}{\lambda} \right)^{k'}  \pod{\ell},
\end{align*}
where
\begin{align*}
 & \left( \frac{(\lambda s + 1)^2 - \epsilon \lambda^2}{\lambda} \right)^{k'} \\
 = & \left( \frac{\lambda ^2 s^2 + 2 \lambda s + 1 - \epsilon \lambda^2}{\lambda}\right)^{k'}  \\
 = & (\lambda s^2 + 2s + \lambda^{-1} - \epsilon \lambda)^{k'} \\
 = & (\lambda (s^2 - \epsilon) + 2s + \lambda^{-1})^{k'}.
\end{align*}
Here, we just need the constant terms of $\left(\frac{(\lambda s +1)^2 - \epsilon \lambda^2}{\lambda} \right)^{k'}$ as the other terms are powers of $\lambda$, and the sum of these powers is zero modulo $\ell$. Now, 
\begin{align*}
 & \text{constant term of } (\lambda (s^2 - \epsilon) + 2s + \lambda^{-1})^{k'} \\
 \equiv  & \sum_{i=0} ^{\frac{k'}{2}} \frac{k' !}{i! i! (k' - 2i)!}(2s)^{k' - 2i}(s^2 - \epsilon)^i \pod{\ell} \quad \text{for $k'$ even}.
\end{align*}
Thus, we obtain that
\begin{align*}
& \sum_{s=1} ^{\ell-1} \alpha_s \sum_{\lambda=1}^{\ell-1} \left(\frac{(\lambda s+ 1)^2 - \epsilon \lambda^2}{\lambda} \right)^{-k} \equiv      & \sum_{s=1} ^{\ell-1} \alpha_s \sum_{i=0}^{\myfloor{\frac{k'}{2}}}\frac{k' !}{i! i! (k' - 2i)!}(2s)^{k' - 2i}(s^2 - \epsilon)^i \pod{\ell}.
\end{align*}

If $\alpha_s \equiv 1\pod\ell$ for $s = 1, \ldots, \ell-1$, then for $k' > 0$ even, we have that
\begin{align*}
& \sum_{s=1} ^{\ell-1} \alpha_s \sum_{i=0}^{\myfloor{\frac{k'}{2}}} \frac{k' !}{i! i! (k' - 2i)!}(2s)^{k' - 2i}(s^2 - \epsilon)^i \\
\equiv &  \sum_{s=1} ^{\ell-1} \sum_{i=0}^{\frac{k'}{2}} \frac{k' !}{i! i! (k' - 2i)!}(2s)^{k' - 2i}(s^2 - \epsilon)^i \\
\equiv &   \sum_{i=0}^{\frac{k'}{2}} \sum_{s=1} ^{\ell-1} \frac{k' !}{i! i! (k' - 2i)!}(2s)^{k' - 2i}(s^2 - \epsilon)^i \\
\equiv & \epsilon^{\frac{k'}{2}} \frac{k'!}{\frac{k'}{2}! \frac{k'}{2}!} \not\equiv 0 \pod{\ell}.
\end{align*}

The last equality holds because the only power of $s$ whose exponent is divisible by $\ell-1$ happens when $i = k'/2$. This proves the lemma.
\end{proof}

\begin{lem}
\label{lemB}
Let $\alpha_s \equiv1 \pod\ell$ for $s \in \f_\ell^\times$, then the sum
\begin{align}
 \sum_{s=1} ^{\ell-1} \alpha_s \sum_{\lambda=1}^{\ell-1} \left(\frac{\lambda}{(\lambda s + 1)^2 - \epsilon \lambda ^2} \right)^{k},
\end{align}
is equal to zero modulo $\ell$ for $k$ odd.
\end{lem}

\begin{proof}
We choose $k' \in \mathbb{N}$ such that $k \equiv -k' \pod {\ell-1}$, and $1 \le k' \le \ell -2$. Then
\begin{align*}
 & \sum_{s=1} ^{\ell-1} \alpha_s \sum_{\lambda=1}^{\ell-1} \left(\frac{(\lambda s+ 1)^2 - \epsilon \lambda^2}{\lambda} \right)^{-k} \\
 \equiv & \sum_{s=1} ^{\ell-1} \alpha_s \sum_{\lambda=1}^{\ell-1} \left(\frac{(\lambda s + 1)^2 - \epsilon \lambda^2}{\lambda} \right)^{k'} \pod{\ell}.
\end{align*}
Hence, we get
\begin{align*}
 & \left( \frac{(\lambda s + 1)^2 - \epsilon \lambda^2}{\lambda} \right)^{k'}  \\
 = & \left( \frac{\lambda ^2 s^2 + 2 \lambda s + 1 - \epsilon \lambda^2}{\lambda}\right)^{k'}  \\
 = & (\lambda s^2 + 2s + \lambda^{-1} - \epsilon \lambda)^{k'} \\
 = & (\lambda (s^2 - \epsilon) + 2s + \lambda^{-1})^{k'}.
\end{align*}
Therefore, we have
\begin{align*}
 & \text{constant term of } (\lambda (s^2 - \epsilon) + 2s + \lambda^{-1})^{k'} \equiv\\
 & \sum_{i=0} ^{\frac{k' - 1}{2}} \frac{k'!}{i! i! (k' - 2i)!}(2s)^{k' - 2i}(s^2 - \epsilon)^i \pod\ell\quad \text{for $k'$ odd}
\end{align*}
Thus,
\begin{align*}
& \sum_{s=1} ^{\ell-1} \alpha_s \sum_{\lambda=1}^{\ell-1} \left(\frac{(\lambda s+ 1)^2 - \epsilon \lambda^2}{\lambda} \right)^{-k} \equiv       & \sum_{s=1} ^{\ell-1} \alpha_s \sum_{i=0}^{\myfloor{\frac{k'}{2}}}\frac{k' !}{i! i! (k' - 2i)!}(2s)^{k' - 2i}(s^2 - \epsilon)^i \pod{\ell}.
\end{align*}

If $\alpha_s \equiv 1\pod\ell$ for $s = 1, \ldots, \ell-1$, then for $k' > 0$ odd, we have that
\begin{align*}
& \sum_{s=1} ^{\ell-1} \alpha_s \sum_{i=0}^{\myfloor{\frac{k'}{2}}} \frac{k' !}{i! i! (k' - 2i)!}(2s)^{k' - 2i}(s^2 - \epsilon)^i \\
\equiv &  \sum_{s=1} ^{\ell-1} \sum_{i=0}^{\frac{k'-1}{2}} \frac{k' !}{i! i! (k' - 2i)!}(2s)^{k' - 2i}(s^2 - \epsilon)^i \\
\equiv &  \sum_{i=0}^{\frac{k'-1}{2}} \sum_{s=1} ^{\ell-1} \frac{k' !}{i! i! (k' - 2i)!}(2s)^{k' - 2i}(s^2 - \epsilon)^i \\
\equiv & 0 \pod{\ell}.
\end{align*}

The last equality holds because there are no powers of $s$ whose exponent is divisible by $\ell-1$, which proves the lemma.
\end{proof}

\begin{lem}
\label{lemC}
Let $\beta_s \equiv s^{-1} \pod \ell$ for $s\in \f_\ell^{\times}$. Then, the sum
\begin{align}
 \sum_{s=1} ^{\ell-1} \beta_s \sum_{\lambda=1}^{\ell-1} \left(\frac{\lambda}{(\lambda s + 1)^2 - \epsilon \lambda ^2} \right)^{k},
\end{align}
is zero modulo $\ell$ for $k$ even.
\end{lem}
\begin{proof}
In the case $k = 0$, we cannot use a binomial expansion so we perform a direct computation:
\begin{align}
\label{casezero}
 & \sum_{s=1} ^{\ell-1} \beta_s \sum_{\lambda=1}^{\ell-1} \left(\frac{\lambda}{(\lambda s + 1)^2 - \epsilon \lambda ^2} \right)^{k} \equiv \sum_{s=1}^{\ell-1} \beta_s (\ell-1) \equiv \sum_{s=1}^{\ell-1} s^{-1} (\ell-1)  \notag \\
\equiv & \sum_{s=1}^{\ell-1} s^{\ell-2}(\ell-1) \equiv (\ell-1)\sum_{s=1}^{\ell-1} s^{\ell-2} \equiv 0 \pod{\ell}.
\end{align}
For $k' > 0$ even, we have that
\begin{align*}
& \sum_{s=1} ^{\ell-1} \beta_s \sum_{i=0}^{\myfloor{\frac{k'}{2}}} \frac{k' !}{i! i! (k' - 2i)!}(2s)^{k' - 2i}(s^2 - \epsilon)^i \\
\equiv &\sum_{s=1} ^{\ell-1} s^{-1} \sum_{i=0}^{\frac{k'}{2}} \frac{k' !}{i! i! (k' - 2i)!}(2s)^{k' - 2i}(s^2 - \epsilon)^i \pod{\ell} \\
\equiv & \sum_{i=0}^{\frac{k'}{2}} \sum_{s=1} ^{\ell-1} s^{-1} \frac{k' !}{i! i! (k' - 2i)!}(2s)^{k' - 2i}(s^2 - \epsilon)^i \pod{\ell} \equiv 0 \pod{\ell}.
\end{align*}

The last equality holds because there are no powers of $s$ whose exponent is divisible by $\ell-1$, which proves the lemma.
\end{proof}

\begin{lem}
\label{lemD}
Let $\beta_s \equiv s^{-1} \pod \ell$ for $s\in \f_\ell^{\times}$. Then the sum
\begin{align}
 \sum_{s=1} ^{\ell-1} \beta_s \sum_{\lambda=1}^{\ell-1} \left(\frac{\lambda}{(\lambda s + 1)^2 - \epsilon \lambda ^2} \right)^{k}
\end{align}
is non-zero modulo $\ell$ for $k$ odd.
\end{lem}
\begin{proof}
For $k' > 0$ odd, we have that
\begin{align*}
& \sum_{s=1} ^{\ell-1} \beta_s \sum_{i=0}^{\myfloor{\frac{k'}{2}}} \frac{k' !}{i! i! (k' - 2i)!}(2s)^{k' - 2i}(s^2 - \epsilon)^i \\
\equiv &\sum_{s=1} ^{\ell-1} s^{-1} \sum_{i=0}^{\frac{k'-1}{2}} \frac{k' !}{i! i! (k' - 2i)!}(2s)^{k' - 2i}(s^2 - \epsilon)^i  \pod{\ell} \\
\equiv &\sum_{i=0}^{\frac{k'-1}{2}} \sum_{s=1} ^{\ell-1} s^{-1} \frac{k' !}{i! i! (k' - 2i)!}(2s)^{k' - 2i}(s^2 - \epsilon)^i \pod{\ell} \\
\equiv & \epsilon^{\frac{k'-1}{2}} \frac{k'!}{\frac{k'-1}{2}!\frac{k'-1}{2}!} \not\equiv 0 \pod{\ell}.
\end{align*}

The last equality holds because the only power of $s$ whose exponent is divisible by $\ell-1$ happens when $i = \frac{k'-1}{2}$, which proves the lemma.
\end{proof}

\begin{cor}
The operator $\sum_{s=1} ^{\ell-1} (\alpha_s+\beta_s) H_s$ has non-zero eigenvalue modulo $\ell$ for all $k > 0$ in its circulant determinant formula.
\end{cor}
\begin{proof}

Using \eqref{casezero}, the operator $\sum_{s=1} ^{\ell-1} (\alpha_s + \beta_s) H_s$ has non-zero eigenvalue for $k = 0$. Furthermore, by Lemmas~\ref{lemA}, ~\ref{lemB},~\ref{lemC}, and ~\ref{lemD}, the eigenvalue of $\sum_{s=1} ^{\ell-1}(\alpha_s + \beta_s) H_s$ is non-zero modulo $\vartheta$ for $k > 0$, since the eigenvalue of $\sum_{s=1} ^{\ell-1} (\alpha_s + \beta_s) H_s$ for $k$ is the sum of the eigenvalues for $k$ of $\sum_{s=1} ^{\ell-1} \alpha_s H_s$ and $\sum_{s=1} ^{\ell-1} \beta_s H_s$.
\end{proof}
The above corollary shows that determinant of $\sum_{s=1} ^{\ell-1} (\alpha_s + \beta_s) H_s$ is non-zero modulo $\ell$, and is hence non-zero. This concludes the proof of Theorem~\ref{main-C}.

\section{Relations between Jacobians of certain modular curves}
\label{lastchapter}

In this section, we summarize some applications of the main results
of this paper to Jacobians of modular curves.

Let $X = X(\ell)$ denote the modular curve of full level $\ell$ structure
which has the structure of a projective algebraic curve over $\q$
for $p \ge 3$ (cf.\ \cite[p.241]{mazur-wiles} or \cite{katz-mazur}).

The group $G = \GL_2(\f_\ell)$ acts on $X$ and the quotients $X_H :=
X/H$ by subgroups $H$ of $G$ (which contain $-1$) exist as
projective algebraic curves over $\q$ \cite[p.244]{mazur-wiles} and \cite{katz-mazur}.

Let $J$ denote the Jacobian of $X$ and $J_H$ denote the Jacobian of
$X_H$.

\begin{prop}
\label{induce-Jacobians} Let $\sigma : \z[G/H'] \rightarrow \Z[G/H]$
be a $\z[G]$-module homomorphism. Then $\sigma$ induces a
homomorphism of Jacobians $\sigma^* : J_H \rightarrow J_{H'}$.
\end{prop}
\begin{proof}
This is proved in \cite[Lemma 3.3]{chen}.
\end{proof}

\begin{prop}
\label{induce-relation} Suppose a cochain complex of $\z[G]$-modules
\begin{equation*}
 \ldots \longrightarrow \z[G/H_{i-1}] \longrightarrow \z[G/H_i] \longrightarrow
 \z[G/H_{i+1}] \longrightarrow \ldots
\end{equation*}
has finite cohomology groups. Then the induced sequence of Jacobians
by applying Proposition~\ref{induce-Jacobians} yields a chain
complex
\begin{equation*}
  \ldots \longleftarrow J_{H_{i-1}} \longleftarrow J_{H_i} \longleftarrow
  J_{H_{i+1}} \longleftarrow \ldots
\end{equation*}
with finite homology groups.
\end{prop}

\begin{proof}
This is proved in \cite[Proposition 3.7]{chen}.
\end{proof}

Theorems~\ref{main-N} and \ref{main-C} imply that
\begin{align}
\label{Q-rel-1} \q[G/N] \longrightarrow_{\psi^+} \q[G/N'] \longrightarrow 0 \\
\label{Q-rel-2} \q[G/C] \longrightarrow_{\psi} \q[G/C']
\longrightarrow 0
\end{align}
are exact cochain complexes of $\q[G]$-modules.

\begin{prop} The following are cochain complexes
\begin{align}
   \z[G/N] \longrightarrow_{\psi^+} \z[G/N'] \longrightarrow 0 \\
   \z[G/C] \longrightarrow_{\psi} \z[G/C'] \longrightarrow 0
\end{align}
with finite cohomology groups.
\end{prop}
\begin{proof}
This follows from tensoring the cochain complexes above by $\q$. If
the cohomology groups were not finite, this would contradict the
exactness of the cochain complexes in
\eqref{Q-rel-1}-\eqref{Q-rel-2}.
\end{proof}

Applying Proposition~\ref{induce-relation}, we obtain:
\begin{cor}
\label{main-cor}
The following are chain complexes
\begin{align}
  0 \longrightarrow J_{N'} \longrightarrow_{\psi^{+*}} J_N \\
  0 \longrightarrow J_{C'} \longrightarrow_{\psi^*} J_C
\end{align}
with finite homology groups.
\end{cor}
From \cite{chen-98}, we have that
\begin{align}
    J_N \sim J_{N'} \times J_B \\
    J_C \sim J_{C'} \times J_B^2,
\end{align}
where $\sim$ denotes the relation of isogeny over $\q$, and $B$ is the subgroup of upper triangular matrices in $G$. Hence, Corollary~\ref{main-cor} describes the main part of the well-known relations between $J_N$ and $J_{N'}$ (resp.\ $J_C$ and $J_{C'}$) using explicit correspondences. 

It is known that $X_{C} \cong X_0(\ell^2)$ and $X_{N} \cong X_0(\ell^2)/\left< w_\ell \right>$, which are the more standard modular curves studied in the literature.


\begin{thebibliography}{10}

\bibitem{chen}
I.~Chen.
\newblock On relations between Jacobians of certain modular curves,
\newblock {\em J. Algebra}, 231 (2000), 414--448.

\bibitem{chen-98}
I.~Chen.
\newblock The Jacobians of non-split Cartan modular curves. 
\newblock Proc. London Math. Soc. (3) 77 (1998), no. 1, 1--38.

\bibitem{edixhoven} 
B.~ De Smit and B.~Edixhoven
\newblock Sur un r\'esultat d'Imin Chen. (French) [On a result of Imin Chen] 
\newblock Math. Res. Lett. 7 (2000), no. 2-3, 147--153.


\bibitem{merel} 
H.~Darmon and L.~Merel 
\newblock Winding quotients and some variants of Fermat's Last Theorem.
\newblock J. Reine Angew. Math. 490 (1997), 81--100.


\bibitem{mazur}
B.~Mazur.
\newblock Rational isogenies of prime degree,
\newblock {\em Invent. Math.} 44 (1978), 129--162.

%\bibitem{serre}
%J.-P.~Serre.
%\newblock Propri\'et\'es galoisiennes des points d'ordre fini des courbes elliptiques,
%\newblock {\em Inventiones Math.} 15 (1972), 259--331.

\bibitem{zagier}
B.\ Birch and D.~Zagier.
\newblock Personal communication with I. Chen, 2000.

\bibitem{katz-mazur}
N.-M.~Katz and B.~Mazur.
\newblock Arithmetic moduli of elliptic curves,
\newblock {\em Princeton University Press}, 1985.

%\bibitem{swinnerton-dyer}
%P.~Swinnerton-Dyer.
%\newblock Analytic theory of abelian varieties,
%\newblock London Mathematical Society Lecture Note Series (Book 14),
%\newblock {\em Cambridge University Press}, 1974.

\bibitem{mazur-wiles}
B.~Mazur and A.~Wiles.
\newblock Class fields of abelian extensions of $\mathbb{Q}$.
\newblock {\em Inventiones Mathematicae} 76 (1984), 179-330.

\bibitem{Neukirch}
J.~Neukirch.
\newblock Algebraic Number Theory,
\newblock {\em Springer-Verlag}, New York, 1999.
\end{thebibliography}
\end{document}